\documentclass[12pt,twoside,singlespace]{article}
\usepackage{array, paralist, enumerate, amsmath, amsthm, amsfonts, amssymb, color, mathrsfs,comment}

\usepackage{geometry}
\usepackage{framed}
\usepackage{hyperref}
\usepackage{graphicx}
\usepackage{epstopdf}

\usepackage{tikz}
\usepackage{tkz-graph}
\usetikzlibrary{arrows,%
                shapes,positioning}

\definecolor{DarkBlue}{rgb}{0,0,0.8} 
\definecolor{DarkGreen}{rgb}{0,0.5,0.0} 
\definecolor{DarkRed}{rgb}{0.9,0.0,0.0} 
\tikzset{
    >=stealth',
    punkt/.style={
           rectangle,
           rounded corners,
           draw=black, very thick,
           text width=6.5em,
           minimum height=2em,
           text centered},
    down/.style={
           regular polygon,
           regular polygon sides=3,
           draw,
           minimum size=0.5cm},
    up/.style={
           regular polygon,
           regular polygon sides=3,
           draw,
           regular polygon rotate=60,
           minimum size=0.5cm},
    all/.style={
           regular polygon,
           regular polygon sides=4,
           draw,
           minimum size=0.7cm},
    none/.style={
           circle,
           draw,
           minimum size=0.5cm},
    simple/.style={
           circle,
           draw,
           fill=black,
           inner sep=0pt,
           minimum width=4pt},
    pil/.style={
           ->,
           thick,
           shorten <=2pt,
           shorten >=2pt,}
}

\usepackage[T1]{fontenc}
\usepackage[latin1]{inputenc}

\numberwithin{equation}{section}
\newtheorem{thm}[equation]{Theorem}
\newtheorem{lem}[equation]{Lemma}

\newtheorem{prop}[equation]{Proposition}

\theoremstyle{definition}

\newtheorem{alg}{Algorithm}

\newcommand{\ZZ}{\mathbf{Z}}

\newcommand{\RR}{\mathbf{R}}

\newcommand{\on}{\operatorname}

\newcommand{\val}{\on{Val}}

\author{Yan X Zhang \\ Department of Mathematics, UC Berkeley\\Berkeley, CA, USA}
\title{Four Variations on Graded Posets}

\begin{document}

\pagestyle{plain}

\maketitle

\maketitle
\begin{abstract}
We explore the enumeration of some natural classes of graded posets, including all graded posets, $(2+2)$- and $(3+1)$-avoiding graded posets, $(2+2)$-avoiding graded posets, and $(3+1$)-avoiding graded posets. We obtain enumerative and structural theorems for all of them. Along the way, we discuss a situation when we can switch between enumeration of labeled and unlabeled objects with ease, generalize a result of Postnikov and Stanley from the theory of hyperplane arrangements, answer a question posed by Stanley, and see an old result of Klarner in a new light.
\end{abstract}

\section{Introduction}
\label{sec:introduction}

The enumeration of posets has been a classical part of enumerative combinatorics; \emph{graded} posets form a very natural subfamily of posets. $(2+2)$-avoidance and $(3+1)$-avoidance are two types of poset-avoidance frequently encountered in literature. Combining these two types of avoidance gives $4$ natural families of graded posets:

\begin{compactitem}
\item All graded posets;
\item $(2+2)$- and $(3+1)$-avoiding graded posets.
\item $(2+2)$-avoiding graded posets;
\item $(3+1)$-avoiding graded posets.
\end{compactitem}

Our paper summarizes some enumerative results on these four families. After some preliminaries in Section~\ref{sec:preliminaries}, we enumerate all graded posets in Section~\ref{sec:all-graded} in a way that better packages old results of Klarner~\cite{klarner-graded} and Kreweras~\cite{kreweras}. We also set the stage for the other families: graded posets, thanks to their grading, are amenable to the transfer-matrix method. 

Then, we take a detour in Section~\ref{sec:seeds} and explore a question of Stanley, namely to understand the phenomenon that in some poset-counting situations there is a very simple substitution that obtains the ordinary generating function (of unlabeled posets) from the exponential generating function (of labeled posets) and vice-versa. The main result in this Section, Theorem~\ref{thm:seeds}, describes a class of posets (in fact, the notion generalizes quite intuitively to other combinatorial objects) where this phenomenon occurs, namely that the automorphism group of specific members of the class must be trivial, which is guaranteed for us when we study $(2+2)$-avoiding posets. Some other ramifications of these results are discussed at the end of Section~\ref{sec:seeds}, including a generalization of an observation in work of Stanley and Postnikov on hyperplane arrangements \cite{postnikov-stanley}, and Yangzhou Hu's \cite{hu} direct application of our theorem (from when our work was in progress). We also generalize this discussion in Appendix~\ref{app:seeds} with Polya theory.

Finally, we go back to counting posets in Sections~\ref{sec:(3+1)+(2+2)}, \ref{sec:(2+2)}, and \ref{sec:(3+1)}. The work in the first two of these sections is new; furthermore, these sections deal with $(2+2)$-avoiding posets, so Section~\ref{sec:seeds} applies. Section~\ref{sec:(3+1)} is a brief summary of our standalone work of enumerating $(3+1)$-avoiding graded posets joint with Lewis \cite{lewis-zhang} and serves as a nice complement to our other results. A particularly important technique that runs in the background of these sections is Lemma~\ref{lem:locality}, which takes a potentially ``global'' condition of chain-avoidance and shows that it suffices to check ``local'' conditions in our poset. We end with some ideas for future work in Section~\ref{sec:conclusion}.

\section{Preliminaries}
\label{sec:preliminaries}

\subsection{Generating Functions}

When we count combinatorial structures with a concept of ``size,'' such as posets (where a choice of a ``size function'' would be the number of vertices), we frequently consider using a generating function to count the said structure for every size simultaneously. Two types of generating functions usually occur:
\begin{compactitem}
\item the OGF (\emph{ordinary generating function}), usually for unlabeled structures. If there are $o_n$ objects with size $n$, the OGF is 
\[O(x) = \sum_n o_n x^n.\]
\item the EGF (\emph{exponential generating function}), usually for structures labeled by $[n]$ (defined to be \{1, \ldots, n\}). If there are $e_n$ objects with size $n$, the EGF is 
\[E(x) = \sum_n e_n \frac{x^n}{n!}.\]
\end{compactitem}

\begin{table}
\begin{center}
\begin{tabular}{r|c|c}
 & OGF & EGF \\
\hline 
$n$-element sets (antichains) & $\frac{1}{1-x}$ & $e^x$ \\
\hline 
$n$-element lists (chains) & $\frac{1}{1-x}$ & $\frac{1}{1-x}$ \\
\hline
\end{tabular}
\caption{The OGF and EGF for two simple classes of posets.}
\label{table:OGF-EGF}
\end{center}
\end{table}

OGFs and EGFs together tell the story of the symmetries of the combinatorial family and carry different information. An example of two extreme cases, with the same OGF but different EGFs, can be seen in Table~\ref{table:OGF-EGF}.

\subsection{Posets and Poset Avoidance}
In this work, all posets we care about are finite. We represent (and refer to) all such posets by their Hasse diagrams, so we use the standard language of graphs. In particular, for a poset $P$, we use $V(P)$ to denote the vertices and $E(P)$ to denote its edges.

We say a poset $P$ \emph{contains} (\emph{avoids}) another poset $Q$ if we can (cannot) select some vertices $S \subset V(P)$ and some edges of $P$ such that they form an isomorphic poset to $Q$. We are only going to be interested in the case where $Q$ is of the form $(a_1 + a_2 + \cdots + a_k)$, which we use to mean the union of $k$ incomparable chains of lengths $a_1, \ldots, a_k$; in fact, we will basically only see $(2+2)$ or $(3+1)$. A $(2+2)$-avoiding poset is also called an \emph{interval order}, and a $(2+2)$- and $(3+1)$-avoiding poset is called an \emph{semiorder}. Examples are given in Table~\ref{table:avoidance}. A standard reference on posets is Stanley \cite{stanley-ec1}, and a relationship between $(2+2)$-avoidance and the words ``interval order'' can be found in \cite{bogart_2+2}.

\begin{table}
\begin{center}
\begin{tabular}{rccc}
&
\begin{tikzpicture}[scale=0.1]
  \node[simple] (a) at (0,0) {};
  \node[simple] (b) at (0, 30) {};
  \node[simple] (c) at (-10, 10) {};
  \node[simple] (d) at (-10, 20) {};
  \node[simple] (e) at (10, 15) {};
  \path {
    (a) edge (c)
        edge (e)
    (c) edge (d)
    (b) edge (d)
        edge (e)
    }; 
\end{tikzpicture} &
\begin{tikzpicture}[scale=0.1]
  \node[simple, style=red] (b) at (-10, 15) {};
  \node[simple, style=red] (c) at (10, 15) {};
  \node[simple, style=red] (d) at (-10, 30) {};
  \node[simple, style=red] (e) at (10, 30) {};
  \node[simple] (a) at (0, 0) {};
  \path {
    (a) edge (b)
        edge (c)
        };
  \path [style=very thick, red]{
    (b) edge (d)
    (c) edge (e)
    }; 
\end{tikzpicture}
 &
\begin{tikzpicture}[scale=0.1]
  \node[simple] (a) at (-10,0) {};
  \node[simple, style=red] (z) at (10,0) {};
  \node[simple, style=red] (b) at (-10, 15) {};
  \node[simple, style=red] (c) at (10, 15) {};
  \node[simple] (d) at (-10, 30) {};
  \node[simple, style=red] (e) at (10, 30) {};
  \path {
    (a) edge (b)
        edge (c)
    (b) edge (d)
    (d) edge (c)
    (c) edge (e)
    (z) edge (c)
    }; 
  \path [style=very thick, red]{
    (z) edge (c)
    (c) edge (e)
    }; 
\end{tikzpicture} \\
\hline \\
avoids $(2+2)$ & yes & no & yes \\
avoids $(3+1)$ & yes & yes & no \\
\end{tabular}
\caption{Examples of poset avoidance and different notions of ``graded.'' From left to right, we have a poset which is not graded, a weakly graded poset that is not strongly graded, and a strongly graded poset. We then show which ones avoid $(2+2)$ and $(3+1)$.}
\label{table:avoidance}
\end{center}
\end{table}

\section{All Graded Posets}
\label{sec:all-graded}

There is some ambiguity in the literature about the usage of the word ``graded,'' so it is important that we start by being precise. In this work, we define a poset $P$ to be \emph{weakly graded (of rank $M$)} if it is equipped with a \emph{rank function} $r\colon V(P) \to \ZZ$ such that if $x$ covers $y$ then $r(x) - r(y) = 1$. It is implied that the rank function is actually an equivalence class of rank functions where we consider two such functions to be equivalent up to translating all the values by a constant. Furthermore, we define a poset to be \emph{strongly graded}, or simply \emph{graded} (this choice is consistent with, say, Stanley~\cite{stanley-ec1}), if all maximal chains have the same length. Recall Table~\ref{table:avoidance} for a comparison.

For our work, we mostly care about graded posets, as working with them tends to be a bit easier. Often, weakly graded posets can be enumerated in basically the same manner, with some manual tweaking. For graded posets, we can consider all minimal vertices to have rank $0$ and all maximal vertices to have the same rank $M$; we call $M+1$ (the number of ranks) the \emph{height} of the poset. we use the notation $P(0), P(1), \ldots, P(M)$ to denote the sets of vertices at levels $0, 1, \ldots, M$. Define an \emph{edge-level} to be a height-$2$ graded poset; note a graded poset of height $M$ has $M-1$ edge-levels as induced subgraphs of adjacent ranked vertex sets.

Klarner \cite{klarner-graded} gave an enumeration of weakly graded posets (which he called \emph{graded}) and Kreweras \cite{kreweras} gave a similar enumeration of graded posets (which he called \emph{tiered}\footnote{As promised, there is indeed a lot of ambiguity in the literature!}). Both results were presented in form of nested sums of the type
\[
\sum f(n_1, n_2) f(n_2, n_3) \cdots f(n_{i-1}, n_i) g(n_i),
\]
where $f(m,n)$ and $g(n)$ are specific functions and the sum, as noted by Klarner \cite{klarner-graded} himself, ``extends over all compositions $(n_1, \cdots, n_i)$ of $n$ into an unrestricted number of positive parts.'' Our main result in this section, Theorem~\ref{thm:all-graded}, is the most natural reformulation of these results as these sums occur naturally from matrix multiplication. Even though the two forms are mathematically equivalent, we gain two things from this reframing: one, we move the computation into the more flexible territory of linear algebra; two, we can apply the ``transfer-matrix state of mind'' to other classes of graded posets, which is exactly what we do in the remaining sections.

Before we start, we make a simple but very important observation: the subposet induced by two adjacent ranks of a weakly graded poset is just a bipartite graph. Thus, it makes sense to consider the EGF of $\psi_w(m,n) = 2^{mn}$, which counts bipartite graphs with $m$ vertices on bottom and $n$ on top:
\[
\Psi_w(x, y) = \sum_{m,n \geq 0} \psi_w(m,n)  \frac{x^my^n}{m!n!}= \sum_{m,n \geq 0} \frac{2^{mn} x^my^n}{m!n!}.
\]

Similarly, the subposet induced by two adjacent ranks of a graded poset is just a bipartite graph with no isolated vertices. We can count this by inclusion-exclusion on isolated vertices. The EGF enumerating these structures is 
\begin{align*}
\Psi_s(x,y) & = \sum_{m,n \geq 0} \psi_s(m,n) \frac{x^my^n}{m!n!} \\
& = \sum_{m,n} (2^{mn} - {m \choose 1} 2^{(m-1)n} - {n \choose 1} 2^{m(n-1)} + \cdots) \frac{x^my^n}{m!n!} \\
& = \sum_{m,n} \sum_{j=0}^m \sum_{k=0}^n (-1)^{m-j+n-k} 2^{jk} {m \choose j} {n \choose k} \frac{x^my^n}{m!n!} \\
& = \sum_{j,k} 2^{jk} \sum_{m \geq j} \sum_{n \geq k} (-1)^{m-j}(-1)^{n-k} {m \choose j} {n \choose k} \frac{x^my^n}{m!n!} \\
& = \sum_{j,k} 2^{jk} (\frac{x^j}{j!} - \frac{x^{j+1}}{j!1!} + \frac{x^{j+2}}{j!2!} - \cdots) (\frac{y^k}{k!} - \frac{y^{k+1}}{k!1!} + \frac{y^{k+2}}{k!2!} - \cdots) \\
& = \sum_{j,k} 2^{jk} \frac{x^j}{j!} (e^{-x}) \frac{y^k}{k!} (e^{-y}) \\
& = e^{-x-y}\Psi_w(x,y),
\end{align*}
where we used inclusion-exclusion in the second step. $\Psi_w$ will resurface when we talk about $(3+1)$-graded posets in Section~\ref{sec:(3+1)}. With the concession of treating the $\psi_s$'s as ``simple,'' we obtain the following: 
\begin{thm}
\label{thm:all-graded}
For the enumeration of all graded posets, the following hold:
\begin{compactitem}
\item Take $r \geq 1$ an integer. Let $V_r = \begin{bmatrix}\sqrt{x^1/1!} & \sqrt{x^2/2!} & \cdots & \sqrt{x^r/r!} \end{bmatrix}^T$ be an $r \times 1$ matrix. Let $D_r$ be the $r \times r$ diagonal matrix with diagonal entries matching those of $V_r$. Let $A_r = \{A_{ij}\}_{i,j}$ be the $r \times r$ matrix where $A_{ij} = \psi_s(i,j)$ for all $i, j$. Then the EGF for $p_{n,r, k}$, the number of graded posets of size $n$, height $k$, and at most $r$ vertices on each rank, is
\[
P_{r,k}(x) = \sum_n p_{n,r,k} x^n/n! = V_r^T(D_rA_rD_r)^{k-1}V_r.
\]
The EGF for graded posets of size $n$ and at most $r$ vertices on each rank is
\[
P_r(x) = \sum_n p_{n,r} x^n/n! = V_r^T (I_r-D_rA_rD_r)^{-1} V_r.
\]
\item Let $V = \begin{bmatrix}\sqrt{x^1/1!} & \sqrt{x^2/2!} & \cdots & \sqrt{x^r/r!} & \cdots \end{bmatrix}^T$ be the infinite column matrix limit of $V_r$. Similarly, let $D$ and $A$ be infinite matrices defined as the limits of $D_r$ and $A_r$. Then, the exponential generating function for the number of graded posets on $n$ vertices is:
\[
P(x) = V^T (I-DAD)^{-1} V.
\]
\end{compactitem} 
\end{thm}

\begin{proof}
We can think of a graded poset with height $n$ as $(n-1)$ edge-levels (height-$2$ graded posets) glued together at their vertices. There are $A_{ij} = \psi_s(i,j)$ possible edge-levels with $(i+j)$ vertices (we use this notation to mean there are $i$ vertices on the bottom and $j$ on top), so we can keep track of how the slices fit together via the transfer-matrix method (see Stanley~\cite{stanley-ec1} for an overview of this technique).

Consider the exponential generating function $P_{r,k}(x)$ for labeled weakly graded posets of height $k+1$, with no more than $r$ vertices on each rank. As stated, these objects are exactly $k$ edge-levels glued together at their vertices. An edge-level with $(i+j)$ vertices can be glued on the bottom of any edge-level with $(j+k)$ vertices. 

Now, consider a complete graph (with all loops) $G$ with $r$ vertices, labeled $1$ through $r$, with edge weight $\psi_s(i,j)$ on each edge between $i$ and $j$. Now, fix a parition of $a_0 + a_1 + \cdots + a_k$ labelled vertices. The product of edge weights in $G$ on the path of length $k$ of form $(a_0, \ldots, a_k)$ in $G$ is exactly the number of ways to draw the edges of a graded poset $P$ of height $k+1$ having $|P(i)| = a_i$ for each rank $i$. Thus, if we let $A'_r = \{A_{ij} \frac{x^i}{i!}\}_{ij}$ be a $r \times r$ matrix, we have

\begin{align*}
P_{r,k+1}(x) & = \sum_{c_1, \ldots, c_{k+1}} {c_1 + \cdots + c_{k+1} \choose c_1, c_2, \ldots, c_{k+1}}A_{c_1, c_2} \cdots A_{c_k, c_{k+1}} \frac{x^{c_1 + c_2 + \cdots + c_{k+1}}}{(c_1 + c_2 + \cdots c_{k+1})!} \\
& = \sum_{c_1 = 1}^r \frac{x^{c_1}}{c_1!} \sum_{c_2, \ldots, c_{k+1}} \frac{A_{c_1, c_2}x^{c_2}}{c_2!} \cdots \frac{A_{c_{k}, c_{k+1}}x^{c_{k+1}}}{c_{k+1}!} \\
& = \begin{bmatrix}x^1/1! & x^2/2! & \cdots & x^r/r! \end{bmatrix} \cdot 
(A'_r)^k
 \cdot 
\begin{bmatrix}1 \\ 1 \\ 1 \\ \vdots \end{bmatrix}.
\end{align*}
Here, the purpose of the $\frac{x^i}{i!}$ terms is to account for the labeling of the vertices partitioned by rank. 

We can make this expression more symmetric: let $V_r = \begin{bmatrix}\sqrt{x^1/1!} \\ \sqrt{x^2/2!} \\ \cdots \\ \sqrt{x^r/r!} \end{bmatrix}$. Also, let
\[
D_r = \begin{bmatrix}\sqrt{x^1/1!} & 0 & \cdots & 0 \\ 0 & \sqrt{x^2/2!} & \cdots & 0 \\ 
\vdots & \vdots & \ddots & \vdots \\ 
0 & 0 & \cdots & \sqrt{x^r/r!} \end{bmatrix}.
\]
If we define $A_r$ to be the $r \times r$ matrix with $(i,j)$ term $A_{ij}$, then $A'_r = A_rD_r^2$. This means that
\begin{align*}
P_{r,k}(x) & = V_r^TD_r(A_rD_r^2)^k \cdot \begin{bmatrix}1 \\ 1 \\ 1 \\ 1 \end{bmatrix} \\
& = V_r^T(D_rA_rD_r)^k D_r \cdot \begin{bmatrix}1 \\ 1 \\ 1 \\ 1 \end{bmatrix} \\
& = V_r^T(D_rA_rD_r)^kV_r.
\end{align*}

Now, summing over all $k$ gives $P_r(x)$, the generating function for all weakly graded posets with at most $r$ vertices on each rank:
\[
P_r(x) = V_r^T (I_r-D_rA_rD_r)^{-1} V_r.
\]

Not only can this be computed explicitly for any $r$, the limit of all of these matrices as $r \to \infty$ are also well-defined as formal-power-series-valued matrices: if we want the $x^n$ term, we can ignore any rows and columns beyond a certain point since all values afterwards will be divisible by $x^n$. Letting $P_r \rightarrow P$, $A_r \rightarrow A$, etc. gives our desired answer for infinite matrices.
\end{proof}

We remark that this method is quite general. In fact:

\begin{thm}
\label{thm:all-weakly graded}
Theorem \ref{thm:all-graded} holds for weakly graded posets if we replace $\psi_s(i,j)$ by $\psi_w(i,j) = 2^{ij}$.
\end{thm}
\begin{proof}
All steps of the proof are identical, except we need to count the number of edge-levels for weakly graded posets. A weakly graded poset of $(i+j)$ vertices is exactly a subgraph of the complete bipartite graph $K_{i,j}$, so we have exactly $2^{ij}$ choices instead of $\psi_s(i,j)$.
\end{proof}

\section{Seeds and Gardens}
\label{sec:seeds}
As an interlude, we examine four results on enumerating interval orders and semiorders, all by different authors:

\begin{compactitem}
\item (Wine and Freund \cite{wine-freund}) OGF of semiorders: $\frac{1-\sqrt{1-4x}}{2x}.$
\item (Stanley \cite{stanley-arrangements}) EGF of semiorders: $\frac{1-\sqrt{4e^{-x} - 3}}{2(1-e^{-x})}$.
\item (Zagier \cite{zagier}) EGF of interval orders: $\sum_n \prod_{k=1}^n(1-e^{-kx})$.
\item (Bousquet-Melou et al. \cite{bousquet-melou}) OGF of interval orders: $\sum_n \prod_{k=1}^n (1-(1-x)^k)$. 
\end{compactitem}

Note that for both interval orders and semiorders, we have $E(x) = O(1-e^{-x}).$ Our work was motivated by Stanley's question (personal communication) of the unifying reason behind this pattern, which clearly does not hold for arbitrary combinatorial structures (recall Table~\ref{table:OGF-EGF})! 

First, we define some notions that capture the concept of ``cloning'' vertices in a poset; these notions are generalizable to other combinatorial structures, though we focus our statements to posets for clarity in this paper. Given a poset $P$, call two incomparable vertices $x$ and $x'$ in $V(P)$ \emph{exchangeable} if for all $y \in V(P)$, $x < y$ (resp. $x > y$) if and only if $x' < y$ (resp. $x' > y$). Given a poset $P$ and a vertex $x \in V(P)$, we denote by \emph{cloning at $x$} the process that outputs a poset $P'$ with an additional vertex $x'$ exchangeable with $x$. We denote by \emph{decloning at $x$} the process that outputs a poset $P'$ with $x$ removed, given that $x$ was exchangeable with at least one other vertex $x' \in V(P)$. An example of cloning can be seen in Figure~\ref{fig:cloning}. 

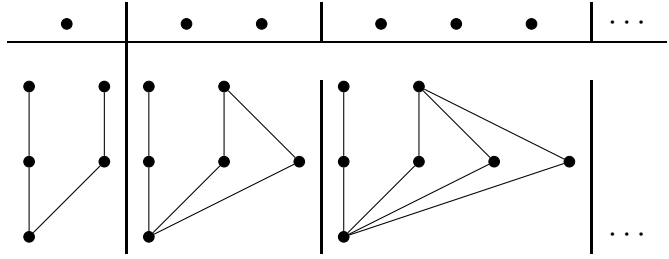
\begin{figure}

\begin{center}
\begin{tabular}{c|c|c|c}
\begin{tikzpicture}[scale=0.1]
  \node[simple] (b) at (10,0) {};
\end{tikzpicture} &
\begin{tikzpicture}[scale=0.1]
  \node[simple] (b) at (10, 0) {};
  \node[simple] (bb) at (20, 0) {};
\end{tikzpicture}
 &
\begin{tikzpicture}[scale=0.1]
  \node[simple] (b) at (10, 0) {};
  \node[simple] (bb) at (20, 0) {};
  \node[simple] (bbb) at (30, 0) {};
\end{tikzpicture}
&
$\cdots$
 \\
\hline \\
\begin{tikzpicture}[scale=0.1]
  \node[simple] (a) at (0,0) {};
  \node[simple] (b) at (10, 0) {};
  \node[simple] (c) at (0, 10) {};
  \node[simple] (d) at (10, 10) {};
  \node[simple] (e) at (0, -10) {};
  \path {
    (a) edge (c)
    (b) edge (d)
    (b) edge (e)
    (a) edge (e)
    }; 
\end{tikzpicture} &
\begin{tikzpicture}[scale=0.1]
  \node[simple] (a) at (0,0) {};
  \node[simple] (b) at (10, 0) {};
  \node[simple] (bb) at (20, 0) {};
  \node[simple] (c) at (0, 10) {};
  \node[simple] (d) at (10, 10) {};
  \node[simple] (e) at (0, -10) {};
  \path {
    (a) edge (c)
    (d) edge (b)
        edge (bb)
    (e) edge (b)
        edge (bb)
    (a) edge (e)
    }; 
\end{tikzpicture}
 &
\begin{tikzpicture}[scale=0.1]
  \node[simple] (a) at (0,0) {};
  \node[simple] (b) at (10, 0) {};
  \node[simple] (bb) at (20, 0) {};
  \node[simple] (bbb) at (30, 0) {};
  \node[simple] (c) at (0, 10) {};
  \node[simple] (d) at (10, 10) {};
  \node[simple] (e) at (0, -10) {};
  \path {
    (a) edge (c)
    (d) edge (b)
        edge (bb)
        edge (bbb)
    (e) edge (b)
        edge (bb)
        edge (bbb)
    (a) edge (e)
    }; 
\end{tikzpicture}
&
$\cdots$
 \\

\end{tabular}
\end{center}

\caption{In each row, cloning one of the vertices repeatedly gives the posets to the right. Decloning at the cloned vertices would revert to the posets at the left. \label{fig:cloning}}
\end{figure}

Let a \emph{seed} be a poset that cannot be decloned (i.e. it has no pair of exchangeable vertices). It is easy to see that repeated decloning associates each poset with a unique seed. Let $P$ be a \emph{sprout} of $Q$ if $P$ can be obtained from $Q$ by a sequence of cloning operations. Let a \emph{garden} $G = G(S)$ be the set of all sprouts of a set $X$ of seeds. To each such $G$ we associate three generating functions:
\begin{compactitem} 
\item the EGF $E(x)$ of labeled elements of $G$;
\item the OGF $O(x)$ of unlabeled elements of $G$;
\item the \emph{seed EGF} $S(x) = \sum_{P} \frac{x^{|V(P)|}}{|V(P)|!}$, where we sum over all labeled seeds $P \in X$. For example, the garden of isolated vertices has a single seed of size $1$; this is just the top row of Figure~\ref{fig:cloning}. 
\end{compactitem}

Call a poset \emph{primitive} if it has no nontrivial automorphism. Note that a primitive poset must be a seed, because otherwise exchanging two exchangeable vertices is a nontrivial automorphism, but the reverse is not true: the poset $(2+2)$ has automorphism group $\ZZ_2$ but has no pairs of exchangeable vertices. This property of $(2+2)$ is actualy crucial to our situation, as we will see very soon in Proposition~\ref{prop:seeds}. Our main result here is that whenever all seeds in a garden are primitive, we can switch between the aforementioned three generating functions in Figure~\ref{fig:trictionary-abstract} with ease:

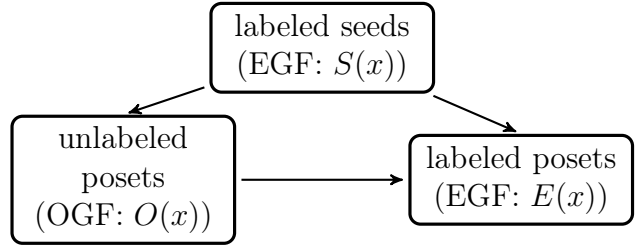
\begin{figure}
\begin{center}
\begin{tikzpicture}[node distance=1cm, auto,]
\node (dummy) {};
\node[punkt, right=of dummy] (e) {labeled posets (EGF: $E(x)$)};
\node[punkt, left=of dummy] (o) {unlabeled posets (OGF: $O(x)$)}
  edge[pil] (e.west);
\node[punkt, above=of dummy] (s) {labeled seeds (EGF: $S(x)$)}
  edge[pil] (o.north)
  edge[pil] (e.north);
\end{tikzpicture}
\end{center}
\caption{The relationship between three families of objects. \label{fig:trictionary-abstract}}
\end{figure}

\begin{thm}
\label{thm:seeds}
When all seeds in a garden are primitive, its OGF $O(x)$, EGF $E(x)$, and seed EGF $S(x)$ satisfy $E(x) = S(e^x-1)$, $E(x) = O(1-e^{-x})$, and $O(x) = S(x/(1-x))$. 
\end{thm}
For sake of smoothness of exposition, we leave the proof to Appendix~\ref{app:seeds}, where we also discuss possible generalizations and connections to Polya theory and cycle indices. For now, it is instructive to note that since every unlabeled primitive seed of size $n$ corresponds to exactly $n!$ labeled primitive seeds, \textbf{for a single primitive seed, the seed OGF is equal to the seed EGF}. Figure~\ref{fig:trictionary-examples} gives some examples of this principle. In this and future figures, the ``$A019590$'' type of notation refers to OEIS \cite{oeis} sequence indices.

\begin{figure}
\begin{center}
\begin{tabular}{c}

\begin{tikzpicture}[node distance=1cm, auto,]
\node (dummy) {};
\node[punkt, right=of dummy] (e) {$e_n$: ordered set partition of $[n]$ (Fubini numbers)};
\node[punkt, left=of dummy] (o) {$o_n$: compositions of $n$ (basically powers of $2$)}
  edge[pil] (e.west);
\node[punkt, above=of dummy] (s) {$s_n$: labeled $n$-chains (permutations) }
  edge[pil] (o.north)
  edge[pil] (e.north);
\end{tikzpicture} \\ 
\begin{tabular}{|l|l|}
\hline
$s_n = A000142: 1, 1, 2, 6, 24, 120, \ldots$ &
 $S(x) = \frac{1}{1-x}$ \\
\hline
$o_n = A011782: 1, 1, 2, 4, 8, 16, 32, \ldots$ & 
$O(x) = \frac{1-x}{1-2x}$ \\
\hline
$e_n = A000670: 1, 1, 3, 13, 75, 501, \ldots$ & 
 $E(x) = \frac{1}{2-e^{x}}$ \\
\hline
\end{tabular} \\[1cm]

\begin{tikzpicture}[node distance=1cm, auto,]
\node (dummy) {};
\node[punkt, right=of dummy] (e) {$e_n$: labeled clones of a single object with total size $n$};
\node[punkt, left=of dummy] (o) {$o_n$: unlabeled clones of a single object of size $n$ (always 1)}
  edge[pil] (e.west);
\node[punkt, above=of dummy] (s) {$s_n$: a single object (always 1)}
  edge[pil] (o.north)
  edge[pil] (e.north);
\end{tikzpicture} \\ 
\begin{tabular}{|l|l|}
\hline
$s_n = A019590: 1, 1, 0, 0, 0, 0, 0, \ldots$ &
 $S(x) = 1 + x$ \\
\hline
$o_n = A000012: 1, 1, 1, 1, 1, 1, 1, \ldots$ & 
$O(x) = \frac{1}{1-x}$ \\
\hline
$e_n = A000012: 1, 1, 1, 1, 1, 1, \ldots$ & 
 $E(x) = e^{x}$ \\
\hline
\end{tabular} 

\end{tabular}
\end{center}
\caption{Examples of gardens with primitive seeds. Top: an unlabeled seed (the $n$-chain) for each $n$, or equivalently $n!$ labeled seeds. Bottom: a single nontrivial seed of a single vertex.  \label{fig:trictionary-examples}}
\end{figure}

Of particular relevance to us, however, is
\begin{prop}
\label{prop:seeds}
If a seed avoids $(2+2)$, then it is primitive. 
\end{prop}
\begin{proof}
Suppose a seed $P$ is not primitive. Take any nontrivial automorphism $g$ and find a maximal element $x \in P$ not fixed by $g$, so $g(x) = x' \neq x$. Now, since we have a seed, $x$ and $x'$ must not be exchangeable. Since nothing lies above $x$ in $P$, this means there is some $y \in P$ such that $x > y$ but $x' \ngtr y$. Let $g(y) = y'$. Since $g$ is an automorphism, we have $x \ngtr y'$ and $x' > y'$. This means $x, y, x', y'$ form a $(2+2)$.
\end{proof}

In particular, the gardens corresponding to interval orders and semiorders are both $(2+2)$-avoiding, which explains the numerology that we observed at the beginning of this chapter and answers Stanley's question; see Figure~\ref{fig:numerology-explained}. For later sections in our work, we now know that for the classes of graded posets that are $(2+2)$-avoiding, we can cheaply obtain the EGFs and OGFs from one another, and we should not expect this nice situation to occur for the ``all graded posets'' and ``$(3+1)$-avoiding graded posets'' classes (indeed, we do not have nice OGFs for them). A few other remarks on $(2+2)$-avoidance follow:
\begin{compactitem}
\item For her work on graded semiorders (which she called \emph{fixed-length semiorders}), Hu \cite{hu} cited our work (when it was in preparation) to switch between labeled and unlabeled enumeration.
\item One of the main results \cite[Theorem~7.1]{postnikov-stanley} in Postnikov and Stanley's work on deformations of Coxeter hyperplane arrangements is as a special case of Theorem~\ref{thm:seeds}.
\item We can use Proposition~\ref{prop:seeds} to immediately generalize our observation about semiorders to that of \emph{marked interval orders} as found in Stanley \cite{stanley-arrangements}. We leave this discussion to Appendix~\ref{sec:marked}.
\item Philosophically, Proposition~\ref{prop:seeds} explains why we would expect $(2+2)$-avoidance to be studied by looking at exchangeability of elements. This seems to have been the approach in the enumeration of $(2+2)$-free posets by indistinguishable elements by Dukes et al. \cite{dukes} 
\item Counting compositions in Figure~\ref{fig:trictionary-examples} can be cutely put into the frame of $(2+2$)-avoidance; compositions are just $(2+1)$-avoiding posets, which happen to be $(2+2)$-avoiding!
\end{compactitem}

\begin{figure}[!t] 
\begin{center}
\begin{tabular}{c}
\begin{tikzpicture}[node distance=1cm, auto,]
\node (dummy) {};
\node[punkt, right=of dummy] (e) {$e_n$: labeled interval orders};
\node[punkt, left=of dummy] (o) {$o_n$: unlabeled interval orders, ascent sequences}
  edge[pil] (e.west);
\node[punkt, above=of dummy] (s) {$s_n$: upper-triangular $0$-$1$ matrices with no zero rows/columns}
  edge[pil] (o.north)
  edge[pil] (e.north);
\end{tikzpicture} \\
\begin{tabular}{|l|l|}
\hline
$s_n = A138265(n):1, 1, 1, 2, 5, 16, 61, \ldots$ &
 $S(x) = \sum_n \prod_k(1-(1+x)^{-k})$ \\
\hline
$o_n = A022493(n): 1, 1, 2, 5, 15, 52, 217, \ldots$ & 
$O(x) = \sum_n \prod_k (1-(1-x)^k)$ \\
\hline
$e_n = A079144(n): 1, 1, 3, 19, 207, 3451, \ldots$ & 
 $E(x) = \sum_n \prod_k (1-e^{-kx})$ \\
\hline
\end{tabular} \\[1cm]
\begin{tikzpicture}[node distance=1cm, auto,]
\node (dummy) {};
\node[punkt, right=of dummy] (e) {$e_n$: labeled semiorders};
\node[punkt, left=of dummy] (o) {$o_n$: unlabeled semiorders (Catalan numbers)}
  edge[pil] (e.west);
\node[punkt, above=of dummy] (s) {$s_n$: labeled Motzkin objects }
  edge[pil] (o.north)
  edge[pil] (e.north);
\end{tikzpicture} \\

\begin{tabular}{|l|l|}
\hline
$s_n = A001006(n) \cdot n!: 1, 1, 2, 12, 96, 1080, \ldots$ &
$S(x) = \cdots$ \\
\hline
$o_n = A000108(n): 1, 1, 2, 5, 14, 42, 132, \ldots$ &
$O(x) = \frac{1-\sqrt{1-4x}}{2x}$ \\
\hline
$e_n = A006531(n): 1, 1, 3, 19, 183, 2371, \ldots$ &
$E(x) = \cdots$ \\
\hline
\end{tabular}
\end{tabular}
\end{center}
\caption{Triplets of generating functions for interval orders (top) and semiorders (bottom). Note the juxtoposition of not-obviously related objects. \label{fig:numerology-explained}}
\end{figure}

\section{$(3+1)$- and $(2+2)$-Avoiding Graded Posets}
\label{sec:(3+1)+(2+2)}

Now that we enter the territory of poset avoidance, we introduce a simple and powerful idea that we use in all the remaining cases. The intuition underneath the technical-sounding description is that the poset-avoidance of sums of chains in graded posets is a ``local condition''; in other words, if a graded poset contains a sum of chains, we do not have to look arbitrarily far to find those chains. 

First, we give an auxiliary definition: given a sum of chains $C = (a_1 + \cdots + a_m)$, we say that $P$ \emph{grade-contains}(resp.\ \emph{grade-avoids}) $a_1[b_1] + \cdots + a_m[b_n]$, where the $b_i$ are nonnegative integers, if we can (resp.\ cannot) find $m$ incomparable chains $C_1, \cdots, C_m$ and a fixed $b$ such that for all $i$, the $i$-th chain has $a_i$ elements, with the $j$-th element of the chain in the $(b_i + j + b)$-th rank of $P$. Thus, the property of containing $a_1[b_1] + \cdots + a_m[b_n]$ is an equivalence condition under adding a constant to all $b_i$. From now on, we slightly abuse notation by just saying ``contains'' (resp.\ ``avoids'') instead of ``grade-contains'' (resp.\ ``grade-avoids'') for brevity. Intuitively, grade-avoidance (containment) captures the idea of avoidance (containment) of chains at specified relative levels whose adjacent elements are adjacent in $P$; for example, avoiding $(2[0]+2[0])$ means avoiding $2$ vertex-disjoint edges (with no other edge between the $4$ vertices) in the same edge-level.

\begin{lem}
\label{lem:locality}
A graded poset $P$ is $(x + y)$-avoiding if and only if it avoids all $(x[b_x] + y[b_y])$ where $\min(b_x, b_y) = 0$, $b_x < y$, and $b_y < x$. 
\end{lem}
\begin{proof}
Suppose $P$ contains an $(x+y)$ somewhere. As grade-containment implies containment, we only need to show that $P$ must also contain a $x[b_x] + y[b_y]$ somewhere with the given inequalities. 

Suppose the $x$-chain has $x$ elements of ranks $b_x < b_x + a_1< b_x + a_2 < \cdots < b_x + a_{x-1}$. This can be extended to a maximal chain where adjacent elements are also adjacent in rank. Since the two elements of rank $b_x$ and $b_x + a_{x-1}$ are both incomparable with all elements of the $y$-chain, all elements between them are also incomparable with the $y$-chain; in particular, we can take the first $x$ adjacent elements at ranks $b_x < b_x + 1 < b_X + 2 < \cdots$, and they will all be incomparable with the $y$-chain. This means we can assume the $x$-chain and (by symmetry) the $y$-chain both appear with elements that are adjacent in rank. Equivalently, $P$ contains some $x[b_x] + y[b_y]$ for some $b_x$ and $b_y$. 

Without loss of generality, we can assume $b_x = 0 < b_y$. Suppose $b_y > x - 1$, meaning that the lowest element of the $y$-chain has higher rank than the highest element of the $x$-chain. Consider any element $s$ covered by the lowest element of the $y$-chain; it is at rank $b_y - 1$, which must be at at least the same rank as any element of the $x$-chain. If it is comparable to any element $r$ in the $x$-chain, we must have $s \geq r$. However, this is impossible as it would mean all the elements of the $y$-chain were greater than $r$, as they are greater than $s$. Thus, we have found a different $y$-chain which has all elements incomparable to the $x$-chain (simply add $s$ to the $y$-chain and remove its top element), which is now ``lower'' in $P$. Repeating this process shows that we can decrease the gap between the two chains to $0$.
\end{proof}

In particular, the set we need to check is finite set since we can assume at least one of the $b_i$'s are $0$. We immediately get that a graded poset $P$ is $(2+2)$-avoiding if and only if it avoids $(2[0] + 2[0])$ and $(2[0] + 2[1])$. Also, $P$ is $(3+1)$-avoiding if and only if it avoids $(3[0] + 1[i])$ for $i$ equal to $0, 1, $ and $2$. Our strategy will now to be to convert this finite list of conditions to a finite list of combinatorial conditions to check. We again give some auxiliary definitions: Let the \emph{up-set} $US(p)$ (resp.\ \emph{down-set $DS(p)$}) of a vertex $p$ be all the vertices that cover $p$ (resp.\ covered by $p$). Call $p$ \emph{up-seeing} (resp.\ \emph{down-seeing}) if $US(p)$ (resp.\ $DS(P)$) is the entire $P(i+1)$ (resp.\ $P(i-1)$). Call a vertex which is both up-seeing and down-seeing \emph{all-seeing}. We obtain:
\begin{lem}
\label{lem:2+2 locality}
A graded poset $P$ is $(2+2)$-avoiding if and only if both of the following conditions are true:
\begin{compactenum}
\item For each level $i$, the up-sets of $P(i)$ (equivalently, the down-sets of $P(i+1)$) form a chain under inclusion (this condition is equivalent to the corresponding edge-level being $(2[0]+2[0])$-avoiding). 
\item Every level of $P$ contains at least one all-seeing vertex.
\end{compactenum}
\end{lem}
\begin{proof}

By definition, $P$ avoids $(2[0] + 2[0])$ if and only if each edge-level of $P$ avoids $(2+2)$. We show that this is equivalent to the up-sets forming a chain. Without loss of generality, consider the edge-level consisting of $P(0) \cup P(1)$. Suppose the up-sets do not form a chain. This means there is some pair of incomparable $US(i)$ and $US(j)$, with $i, j \in P(0)$. Equivalently, there is a pair of elements $i', j \in P(1)$, $i'>i$, $j'>j$, but $i' \ngtr j$ and $j' \ngtr i$. This is exactly a $(2+2)$. These steps are reversible, so we are done. When the up-sets form a chain, the maximal up-set must be the entire level above, else there is some vertex on the upper level that does not have a neighbor on this level, violating the gradedness of $P$. This is equivalent to saying there must be at least one up-seeing vertex. A symmetrical argument shows there must also be at least one down-seeing vertex.

By Lemma~\ref{lem:locality}, $P$ being $(2+2)$-avoiding is equivalent to $P$ avoiding $(2[0] + 2[0])$ and $(2[0] + 2[1])$. Suppose $P$ already avoids $(2[0] + 2[0])$, we want to characterize when $P$ also avoids $(2[0] + 2[1])$. We claim this is equivalent to the second condition.  Take some level $P(i)$. Take $w$ to be a vertex in $P(i+1)$ with the fewest neighbors in $P(i)$. We claim that such a $w$ has the property that if any $v \in P(i)$ is a neighbor of $w$, we must have $v$ be up-seeing. This is because if $v$ is not connected to some $w' \in P(i+1)$, then $DS(w')$ is a proper subset of $DS(w)$, which contradicts the claim that $w$ has the fewest neighbors. Similarly, let $u$ be a vertex in $P(i-1)$ with the fewest neighbors in $P(i)$. Suppose $P(i)$ had no all-seeing vertices. Since it must have at least one up-seeing vertex $x$ and one down-seeing vertex $y$, $x$ must not cover $u$ and $y$ must not be covered by $w$ (else $x$ and/or $y$ would be all-seeing). But $y > u$ and $w > x$, so in order to avoid $(2[0] + 2[1])$ we must have $w > u$ in $P$. This means there is some $w > z > u$, where $z \in P(i)$. However, since $z$ is a neighbor of $w$ and $u$, $z$ must be both up-seeing and down-seeing, a contradiction.

Conversely, if our two conditions hold, we know we have $(2[0] + 2[0])$ by equivalence to the first condition. Furthermore, suppose we have a $(2[0] + 2[1])$ somewhere; say the $3$ relevant ranks are $i$, $i-1$, and $i+1$. Take an all-seeing vertex $v \in P(i)$. It is  comparable to both the higher element in the $2[1]$ and the lower element in the $2[0]$, which is a contradiction as they were supposed to be incomparable.
\end{proof}

\begin{lem}
\label{lem:3+1 locality}
A graded poset $P$ of height $n$ being $(3+1)$-avoiding is equivalent to both of the following conditions being true:
\begin{compactenum}
\item Every vertex $v$ in $P$ is up-seeing and/or down-seeing. 
\item For each $i \leq n-2$, every pair of vertices $v \in P(i)$ and $w \in P(i+2)$ are comparable.
\end{compactenum}
\end{lem}
\begin{proof}
This result can be found in \cite[Theorem 3.1]{lewis-zhang}; the proof is similar in spirit to that of Lemma~\ref{lem:2+2 locality} and not marginally instructive, so we omit it.
\end{proof}

It is refreshing that these two locality conditions synergize:
\begin{thm}
\label{thm:3+1 and 2+2}
A graded poset $P$ is $(3+1)$- and $(2+2)$- avoiding if and only if:
\begin{compactenum}
\item $P$ avoids $(2[0] + 2[0])$;
\item every vertex in $P$ is up-seeing and/or down-seeing;
\item every level of $P$ contains at least one all-seeing vertex.
\end{compactenum}
\end{thm}
\begin{proof}
These conditions are necessary by the previous lemmata. Since the first and the third conditions include the conditions of Lemma~\ref{lem:2+2 locality}, it suffices to show that we can prove the unused condition of Lemma~\ref{lem:3+1 locality}, namely that for each $i$ of height $n-2$ or less, every pair of vertices $v \in P(i)$ and $w \in P(i+2)$ are comparable.

Take $v \in P(i)$, we will show that it is comparable to all $w \in P(i+2)$. Suppose $P(i+2)$ is non-empty (we are vacuously done otherwise).  There is an all-seeing vertex $v' \in P(i+1)$ by the assumption. This means $v'$ is comparable to all $w \in P(i+2)$. However, we also have $v' > v$, so $v$ is comparable to all of $P(i+2)$ as well. 
\end{proof}

We are now equipped to give our main result of this section. The idea is that since there is at least one all-seeing vertex on every level, these vertices form a ``skeleton'' of the poset; furthermore, each vertex not of this form belongs to a unique ``slice'' that spans two adjacent ranks. These two problems can be addressed independently; see Figure~\ref{fig:gis-slices} for intuition. In this and future figures, a flat face in the up (resp.\ down) direction denotes up-seeing (resp down-seeing). Thus, all-seeing vertices are squares, triangles with bottom edge flat are down-seeing, etc.

\begin{figure}
\begin{center}
\begin{tikzpicture}[scale=0.1]
  \node[all] (a) at (-10,0) {};
  \node[down] (b) at (10,0) {};
  \node[all] (i) at (-20,15) {};
  \node[all] (c) at (-10, 15) {};
  \node[down] (d) at (0, 15) {};
  \node[up] (e) at (10, 15) {};
  \node[all] (f) at (-10, 30) {};
  \node[up] (g) at (0, 30) {};
  \node[up] (h) at (10, 30) {};
  \path {
    (c) edge (a)
        edge (b)
        edge (f)
        edge (g)
    (d) edge (a)
        edge (b)
        edge (f)
    (e) edge (f)
        edge (g)
        edge (a)
    (h) edge (c)
        edge (e)
    (i) edge (f)
        edge (g)
        edge (h)
        edge (a)
        edge (b)
    };
  \path [draw=none, fill=gray, semitransparent] {
    (-13,-3) -- (-23, 12) -- (-23, 18) -- (-13,33) -- (-7, 33) -- (-7, -3)};  
  \path [draw=black] {
    (-4,12) -- (-4,33) -- (15, 33) -- (3, 12) -- (-4, 12)};
  \path [draw=black] {
    (6,-3) -- (6,18) -- (14, 18) -- (14, -3) -- (6,-3)};
\end{tikzpicture}
\end{center}
\caption{A height-$k$ graded semiorder can be decomposed into a ``skeleton'' of all-seeing vertices (highlighted in gray) plus $k$ sets of faceoffs (outlined in black).   \label{fig:gis-slices}}
\end{figure}
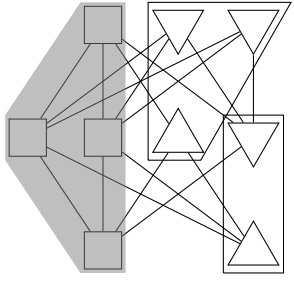

\begin{thm}
\label{thm:graded-interval-semiorders}
The OGF for graded semiorders is
\[
O(x) = \frac{1-3x+2x^2-x^3}{(1-x)(1-3x+x^2)}.
\]
\end{thm}

\begin{proof}

Suppose $P$ has height $k-1$. By Theorem~\ref{thm:3+1 and 2+2}, every level has at least one all-seeing vertex. All other vertices are either just down-seeing or just up-seeing. Define the \emph{faceoff} $T(i)$ be the induced subposet of just down-seeing vertices on $P(i)$ and just up-seeing vertices on $P(i-1)$. Note that the faceoff can be empty, but for each $T(i)$, having a non-zero number of vertices  on one side immediately implies having a nonzero number on the other.

We claim that an unlabeled $P$ is uniquely determined by a nonzero number of all-seeing vertices for each level $P(i)$, $i \leq k-1$, and a choice of a $(2[0] + 2[0])$- avoiding faceoff $T(i)$ for each $1 \leq i \leq k-1$, with no dependency between these choices. By definition, this decomposition is unique given $P$, so it suffices to show that any set of such choices gives a $(3+1)$- and $(2+2)$- avoiding poset. The second and third items in Theorem~\ref{thm:3+1 and 2+2} are immediate. Since the faceoffs are $(2+2)$-avoiding themselves, the first item is also true since no all-seeing vertex can contribute to a $(2+2)$. The main subtlety is that when we select faceoffs for $T(i)$ and $T(i+1)$, we are also implicitly identifying vertices. We must also show that this identification is unique. 

Given this characterization, if we let $T(x)$ be the OGF for faceoffs, it is easy to see that the generating function for graded semiorders of height $k-1$ must be of form $O_k(x) = (\frac{x}{1-x})^k T(x)^{k-1}$, where the first term chooses a nonzero number of all-seeing vertices on each level, and the second term chooses a faceoff for each edge-level; i.e. height-$2$ \textbf{weakly}-graded $(2[0] + 2[0])$-avoiding posets with no all-seeing vertices (these can have isolated vertices, since in $P$ they are actually connected to the all-seeing vertices on the other side!). 

What is $T(x)$? Besides the empty faceoff, we must have at least $2$ vertices, with the rule that no vertex on the bottom is connected to every vertex above and the neighbors of the bottom vertices form a chain. To count the number of faceoffs with $m$ vertices on top and $n$ on the bottom, sort both the top and bottom vertices by degree. Then a vertex on top of degree $d$ must be matched with the $d$ vertices of highest degree on the bottom, and vice-versa; this gives a bijection to Young diagrams of height $m$ and width $n$, giving ${m-1 + n-1 \choose n-1}$ ways. So, we have
\[
T(x) = 1 + \sum_{m\geq 1,n \geq 1} {m-1 + n-1 \choose n-1} x^{m+n} = 1 + x^2 \sum_k 2^kx^k = \frac{(1-x)^2}{(1-2x)}.
\]

Thus, setting $z = \frac{x}{1-x}$, we obtain
\begin{align*}
O(x) & = \sum_{k \geq 0} (z)^k T(x)^{k-1} \\
& = 1 + \frac{z}{1 - zT(x)} \\
& = 1 + \frac{z}{(1 - \frac{x(1-x)}{1-2x})} \\
& = 1 + \frac{z(1-2x)}{1-3x+x^2} \\
& = \frac{x(1-2x) + (1-3x+x^2)(1-x)}{(1-x)(1-3x+x^2)} \\
& = \frac{1 - 3x + 2x^2 - x^3}{(1-x)(1-3x+x^2)}. \qedhere
\end{align*} 
\end{proof}

Theorem \ref{thm:seeds} applies since we have $(2+2)-$ avoidance, so the other generating functions are immediate; the triplet of related generating functions can be found in Figure~\ref{fig:trictionary-graded-interval-semiorders}. Our enumeration of graded semiorders seems to be original (Hu \cite{hu} also enumerates these objects, but in a different way that gives different insights; in particular, she gives a bijection with trees). It is interesting to note that the generating function we obtain also enumerates directed column-convex polyominoes \cite{deutsch}.  Since we have $(2+2)$-avoidance, Proposition~\ref{prop:seeds} applies; looking at the seed EGF suggests some relationship with Fibonacci numbers, as the generating function is somewhat similar to the classic generating function $\frac{1}{1-x-x^2}$. Indeed, the OEIS \cite{oeis} notes that $o_n(x) = F_{2n}+1$, where $F_n$ is the $n$-th Fibonacci number. A direct bijective proof of this fact may be enlightening.

\begin{figure}
\begin{center}
\begin{tikzpicture}[node distance=2cm, auto,]
\node (dummy) {};
\node[punkt, right=of dummy] (e) {$e_n$: labeled graded semiorders};
\node[punkt, left=of dummy] (o) {$o_n$: directed column-convex polyominos (with conditions)}
  edge[pil] (e.west);
\node[punkt, above=of dummy] (s) {$s_n$: variation of Fibonacci numbers }
  edge[pil] (o.north)
  edge[pil] (e.north);
\end{tikzpicture}

\medskip
\begin{tabular}{|l|l|}
\hline
$s_n = A000045(n+1) (n!): 1, 1, 2, 6, 48, 360, \ldots$ &
$S(x) = 1 + x + \frac{x^2}{1-x-x^2}$ \\
\hline
$o_n = A055588(n-1): 1, 1, 2, 4, 9, 22, 56, 145, \ldots$ &
$O(x) = \frac{1-3x+2x^2-x^3}{(1-x)(1-3x+x^2)}$ \\
\hline
$e_n = \textbf{???}: 1, 1, 3, 13, 99, 1021, 12723, \ldots$ &
$E(x) = S(e^x-1)$ \\
\hline
\end{tabular}

\end{center}
\caption{The three generating functions that come up from graded semiorders. The bottom sequence does not appear in the OEIS. \label{fig:trictionary-graded-interval-semiorders}}
\end{figure}
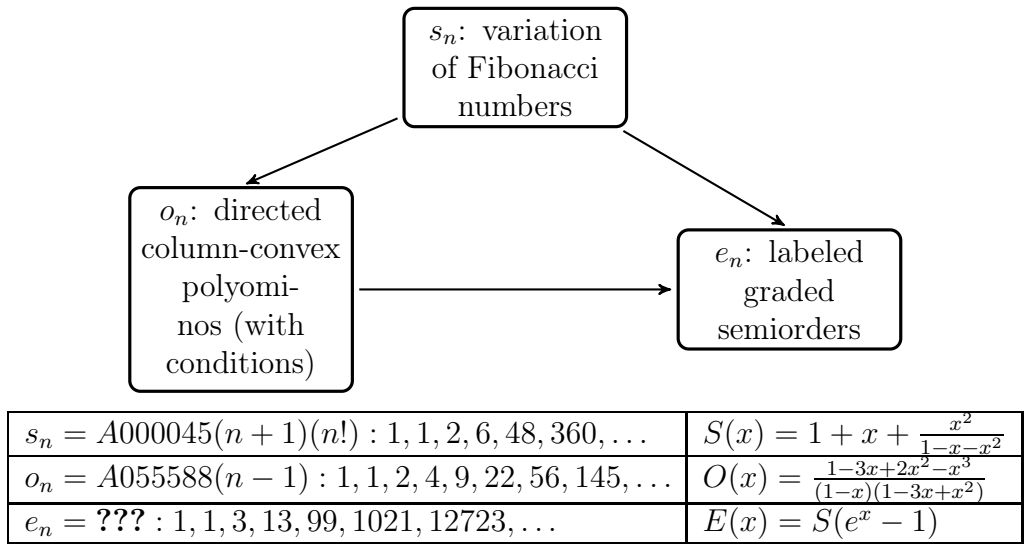

\section{$(2+2)$-Avoiding Graded Posets}
\label{sec:(2+2)}

We now attack graded interval orders, i.e. $(2+2)$-avoiding graded posets. Even though we no longer have $(3+1)$-avoidance when counting $(2+2)$-avoiding graded posets, we still have  Lemma~\ref{lem:2+2 locality} and the transfer-matrix method technique from Section~\ref{sec:all-graded}. We combine these two ideas to count $(2+2)$-avoiding graded posets, with a slightly different $G$ and very different transfer-matrix.

Consider a $(2+2)$-avoiding graded poset $P$ of height $k$.  When we counted all graded posets, we built $P$ by starting with a set of vertices on $P(0)$ (a height-$1$ poset) and iteratively adding a set of vertices and the relevant edges one layer at a time. We only needed to know how many vertices were on $P(i)$ to know how many ways are there to add $P(i+1)$ and the edges in the edge-level induced by $P(i) \cup P(i+1)$. This inspired us to make a graph $G$ whose vertices corresponds to natural numbers $\{1, 2, 3, \ldots \}$, and we counted height-$n$ graded posets by length-$(n-1)$ walks in $G$. For $(2+2)$-avoiding posets this information is no longer sufficient, because we need at least one all-seeing vertex in each row by Lemma~\ref{lem:2+2 locality} and thus we need to know which vertices in the previous level were down-seeing. 

To incorporate this information, we do a few things differently. 
\begin{compactenum}
\item We adjust $V(G)$ to be indexed by \textbf{nonnegative} integers $\{0, 1, 2, \ldots \}$; and their meaning will be different -- instead of being used to denote the number of vertices on the current top level of $P$, they now denote the number of non-down-seeing vertices on the top level (this is also why we allow $0$).
\item As we did in Section~\ref{sec:(3+1)+(2+2)}, we strip away all all-seeing vertices from $P$ and account for them separately. This does not affect $(2[0] + 2[0])$-avoidance, and as long as we know to add back at least one all-seeing vertex on each level, we will in addition get $(2[0] + 2[1])$-avoidance thanks to Lemma~\ref{lem:2+2 locality}. This results in a \textbf{weakly}-graded poset $Q$ (much like in Section~\ref{sec:(3+1)+(2+2)}, when we remove all-seeing vertices we may get isolated vertices on adjacent levels). Once we construct $Q$, we can construct an infinite family of $P$'s by adding at any nonzero number of all-seeing vertices to each level of $Q$.
\end{compactenum}

We now construct a revised iterative procedure to construct $Q$:
\begin{framed}
\begin{alg}
\label{alg:2+2}
Starting at level $i = 0$:
\begin{compactenum}[1]
\item We already have $m$ non-down-seeing vertices in $Q(i)$ (for the first step, we have $m = 0$).
\item Select $l \geq 0$ to be the number of down-seeing but not up-seeing vertices in $Q(i)$ (as we already know the number of non-down-seeing vertices on this level, it is at this step that we finally know many vertices are in $Q(i)$ altogether).
\item Select $n \geq 0$ to be the number of non-down-seeing vertices on level $i+1$ (this $n$ will become $m$ for the next iteration of the algorithm).
\item Choose the edges to exist between the $l + m$ vertices in $Q(i)$ and the $n$ vertices in $Q(i+1)$, such that these vertices and edges create a $(2[0] + 2[0])$-avoiding edge-level and do not create any all-seeing vertices.
\item Increment $i$ by $1$, and go back to step $1$, unless $i$ is now $k-1$, in which case we stop. 
\end{compactenum}
\end{alg}
\end{framed}

See Figure~\ref{fig:2+2 P and Q} for intuition on $Q$ and our algorithm. We use circles to denote vertices which are neither up- nor down-seeing, which are now possible for us.
\begin{figure}
\begin{center}
\begin{tikzpicture}[scale=0.1]
  \node[all] (a) at (-10,0) {};
  \node[down] (b) at (10,0) {};
  \node[down] (l) at (20, 0) {};
  \node[all] (i) at (-20,15) {};
  \node[all] (c) at (-10, 15) {};
  \node[down] (d) at (0, 15) {};
  \node[none] (j) at (10, 15) {};
  \node[none] (k) at (20, 15) {};
  \node[up] (e) at (30, 15) {};
  \node[all] (f) at (-10, 30) {};
  \node[up] (g) at (0, 30) {};
  \node[up] (h) at (10, 30) {};
  \path {
    (c) edge (a)
        edge (b)
        edge (f)
        edge (g)
    (d) edge (a)
        edge (b)
        edge (f)
    (e) edge (f)
        edge (g)
        edge (a)
    (h) edge (c)
        edge (e)
    (i) edge (f)
        edge (g)
        edge (h)
        edge (a)
        edge (b)
    (j) edge (a)
        edge (b)
        edge (f)
    (k) edge (a)
        edge (f)
        edge (b)
    (l) edge (i)
        edge (c)
        edge (d)
    };

  \path [draw=none, fill=gray, semitransparent] {
    (-13,-3) -- (-23, 12) -- (-23, 18) -- (-13,33) -- (-7, 33) -- (-7, -3)};
  \path [draw=black] {
    (-4,12) -- (-4,33) -- (15, 33) -- (25, 12) -- (-4, 12)};
  \path [draw=black] {
    (6,-3) -- (6, 18) -- (35, 18) -- (25, -3) -- (6, -3)};

\end{tikzpicture}
\end{center}
\caption{A height-$k$ graded interval order $P$ can be decomposed into a ``skeleton'' of all-seeing vertices (highlighted in gray) plus a $(2[0] + 2[0])-$ avoiding weakly graded interval order $Q$. The black outlines group $Q$ into layers, each of which is constructed by a single loop of Algorithm~\ref{alg:2+2}. Note the outlines can have nontrivial intersection, as we need to separately determine the up-going and down-going edges for vertices which are neither up- nor down-seeing. \label{fig:2+2 P and Q}}
\end{figure}
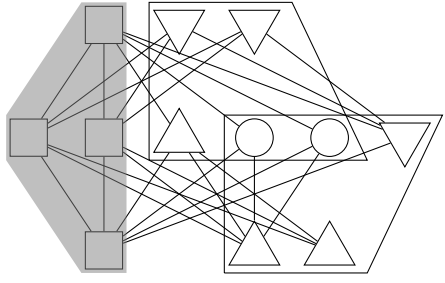

The key observation is that the edge structure is basically completely determined in step $4$; while we will add more vertices to $Q(i+1)$ in the next iteration of the loop, those vertices are all down-seeing and thus we need to make no more choices about their edges. Furthermore, adding more down-seeing vertices to $Q(i+1)$ does not change the $(2[0] + 2[0])$-avoiding status of $Q$. This, combined with the fact that we always have at least one all-seeing vertex on each level, means such a constructed $Q$ (and also the corresponding $P$'s) must be $(2+2)$-avoiding by Lemma~\ref{lem:2+2 locality}. Given any height-$n$ $(2+2)$-avoiding graded poset $P$, by first stripping away all all-seeing vertices and then by looking at which vertices are down-seeing on each level, we can reverse-engineer the choices of vertices and edges that can build such a $P$, so we have a well-defined bijection between $(2+2)$-avoiding posets of height $k$ and the outputs of this construction.

We must then carefully enumerate the choices we make in steps $2$ through $4$ of the algorithm. It is helpful to recall what we did in Section~\ref{sec:all-graded} for counting all graded posets. Suppose we had $m$ vertices on $P(i)$. To account for adding $n$ vertices on $P(i+1)$ and then choosing the edges, we multiplied by  $\psi_s(m,n) \frac{x^n}{n!}$, the $(m,n)$-entry in our matrix. Here, $\psi_s(m,n)$ was the number of strongly graded edge-levels with $m$ and $n$ vertices on the two parts, and $\frac{x^n}{n!}$ kept track of the new vertices. Also, when proving Theorem~\ref{thm:graded-interval-semiorders}, we needed to count (as faceoffs) unlabeled $(2+2)$-avoiding edge-levels, for which we gave a bijection to Young diagrams. We now count labeled $(2+2)$-avoiding edge-levels, which can be seen as a labeled analogue of Young diagrams:

\begin{lem}
\label{lem:2+2 no-all-seeing}
The number of weakly graded $(2+2)$-avoiding height-$2$ posets $Q$ on $m$ distinguishable vertices in $Q(1)$ and $n$ distinguishable vertices in $Q(0)$, with no all-seeing vertices, is \[
\psi_{2+2}(m,n) = \sum_{j = 0}^\infty  (j!)^2 S(n, j) S(m, j),
\] where $S(n,k)$ is the Stirling number of the second kind. It has the generating function \[
\Psi_{2+2}(x, y) = \sum_{m,n} \psi_{2+2}(m,n) \frac{x^m y^n}{m!n!} = \frac{1}{e^x + e^y - e^{x+y}}.
\]
\end{lem}
\begin{proof}
Here, $(2+2)$-avoidance is equivalent to $(2[0] + 2[0])$-avoidance. Recall that $(2[0] + 2[0])$-avoidance is equivalent to the up-sets of $Q(0)$ forming a chain under inclusion. Suppose the up-sets of $Q(0)$ take on exactly $j$ distinct values $S_1 \subset S_2 \subset \cdots \subset S_j$. Note that $S_j \neq Q(1)$ is equivalent to $Q(0)$ having no up-seeing vertex, and $S_1 = \emptyset$ is equivalent to $Q(1)$ having no down-seeing vertex. Thus, the number of ways to pick the values of all the $S_i$ is exactly the number of ways to partition $Q(1)$ into an ordered list of $j$ nonempty sets $T_1 \cup T_2 \cup \cdots \cup T_j$, as we can then define $S_i = T_1 \cup T_2 \cup \cdots \cup T_{i-1}$, with $S_1 = \emptyset$, $S_j = T_1 \cup \cdots \cup T_{j-1}$, and $T_j$ of the elements of $Q(1)$ not belonging to any $S_i$. There are $(j!) S(n, j)$ ways to make this partition. After the values of $S_i$ are fixed, it is an independent problem to assign the up-sets of $Q(0)$ to these values. We need each value to be hit at least once, so this is, familiarly, a partition of $Q(0)$ into an ordered list of $j$ nonempty sets.  There are $(j!) S(l + m', j)$ ways to achieve this. Summing gives our formula. Note that we have $\psi_{2+2}(0,0) = 1$ as $S(0,0) = 1$.

For the generating function, we make the computation

\begin{align*}
\Psi_{2+2}(x, y) = & \sum_{m,n} \psi_{2+2}(m,n) \frac{x^m y^n}{m!n!} \\
= & \sum_{m,n} \sum_{j}   S(n, j) S(m, j) \frac{x^m y^n}{m!n!} \\
= & \sum_j (j!)^2 (\sum_m S(m,j) \frac{x^m}{m!}) (\sum_n S(n,j) \frac{y^n}{n!}) \\
= & \sum_j (j!)^2 \frac{(e^x - 1)^j}{j!} \frac{(e^y-1)^j}{j!} \\
= & \sum_j [(e^x - 1)(e^y-1)]^j \\ 
= & \frac{1}{1 - 1 + e^x + e^y - e^{x+y}} \\
= & \frac{1}{e^x + e^y - e^{x+y}},
\end{align*}
where in the fourth line we used the generating function for Stirling numbers of the second kind
\[
\sum_m S(m,j) x^m/m! = (e^x-1)^j/j!. \qedhere
\]
\end{proof}

Again, taking the liberty of treating $\psi_{2+2}(m,n)$ as ``simple'' as we did with $\psi_S(m,n)$, we can use the transfer-matrix method to achieve the following result similar to Theorem~\ref{thm:all-graded}. 

\begin{thm}
\label{thm:graded-interval-orders}
Let $V$ be the infinite column vector $\begin{bmatrix} 1 & 0 & 0  & \cdots \end{bmatrix}^T$. Let $A$ be the infinite matrix with rows and columns indexed by $i, j \geq 0$, such that 
\[
A_{mn} = A(m, n) = \sum_{l = 0}^\infty \frac{x^{l + n}}{l!n!} \sum_{m' = 0}^m {m \choose m'} \psi_{2+2}(n, l+m').
\]
Then the EGF for $p_{n, k}$, the number of graded interval orders of size $n$ and height $k$, is
\[
P_{k}(x) = \sum_n p_{n,k} x^n/n! = (e^x-1)^k [V^T A^kV] = (e^x - 1)^k [A^k]_{0,0}.
\]
The EGF for graded interval orders of size $n$ is 
\[
P(x) = (e^x - 1) [V^T (I - (e^x - 1)A)^{-1} V] = (e^x - 1) [(I - (e^x - 1)A)^{-1}]_{0,0}.
\]
\end{thm} 
\begin{proof}

The idea is basically the same as the proof of Theorem~\ref{thm:all-graded}, with a couple of differences. First, when counting all graded posets we could start with any number of vertices on the first row. But when counting graded interval orders we must have $0$ non-down-seeing vertices on the first row, so the initial row vector $V$ is different.

When we use the transfer-matrix method to count possible constructions of $Q$, suppose we have $m$ non-down-seeing vertices on $Q(i)$ and wish to have $n$ non-down-seeing vertices on $Q(i+1)$. Suppose exactly $m'$ of the $m$ non-down-seeing vertices in $Q(i)$ are non-up-seeing. There are ${m \choose m'}$ ways to choose this set. We now have exactly $l + m'$ vertices in $Q(i)$ that are non-up-seeing and $n$ vertices in $Q(i+1)$ that are non-down-seeing.  Lemma~\ref{lem:2+2 no-all-seeing} allows us to compute $\psi_{2+2}(n, l+m')$. 

Note it is completely legal to have no non-down-seeing vertices to $Q(i+1)$, as long as we have no non-up-seeing vertices on $Q(i)$; this happens in the sum when we pick $l = 0$ and $m' = 0$, in which case we get a contribution of $\psi_{2+2}(0,0)$. This is why have the summation index in its definition start at $0$ instead of $1$, since we conveniently obtain $\psi_{2+2}(0,0) = 1$. We now obtain:
\begin{align*}
P_{k}(x) & = \begin{bmatrix}1 & 0 & 0 & \cdots \end{bmatrix} \cdot 
A^{k-1}
 \cdot 
\begin{bmatrix}1 \\ 1 \\ 1 \\ \vdots \end{bmatrix} (e^x - 1)^{k}.
\end{align*}
The reason we need to multiply by $(e^x - 1)^{k}$ at the end is because we must add at least one all-seeing vertex to each level to make $P$ from $Q$, corresponding to multiplication by $(x + x^2/2! + x^3/3! + \cdots) = e^x-1$  for each of the $k+1$ levels. 

A slight simplification that adds some symmetry is to notice that $A(m, 0) = 1$ for all $m$. This is because we need both $l$ and $m'$ to be zero for the $\psi_{2+2}$ term to be nonzero in $A(m,n)$. Thus, the all-ones column vector can also be thought of as $AV$. This means we can rewrite
\[
P_k(x) = V^T A^k V (e^x-1)^k
\]
as desired. The rest is routine algebra.
\end{proof}

Using this, we can fairly quickly obtain terms for the EGF, even though we do not have a nice closed-form expression like in the previous section. We can then apply Theorem~\ref{thm:seeds} to obtain the OGF and the seed EGF for free. See Figure~\ref{fig:trictionary-graded-interval-orders} for the result. Sadly, these sequences do not seem to be in the OEIS, even though many extremely similar impostors lurk there. Maybe a keen reader could find insight we did not!  It is instructive to compare with graded semiorders; for example, it is easy to check by hand that there is exactly one $(2+2)$-avoiding graded poset of size $6$ that contains $(3+1)$, accounting for the first difference ($56 + 1 = 57$) at $o_6$, which in turn contributes $6! = 720$ to the first difference ($12723 + 720 = 13443$) in $e_6$ as the poset happens to be a (primitive) seed.

\begin{figure}
\begin{center}
\begin{tabular}{|l|}
\hline
$s_n = \textbf{???}: 1, 1,1,1,2,3,6,12,28,69, \ldots$  \\
\hline
$o_n = \textbf{???}: 1, 1, 2, 4, 9, 22, 57, 155, 442, \ldots$ \\
\hline
$e_n = \textbf{???}: 1, 1, 3, 13, 99, 1021, 13443, \ldots$ \\
\hline
\end{tabular}

\end{center}
\caption{The three sequences that arise from graded interval orders. \label{fig:trictionary-graded-interval-orders}}
\end{figure}

\section{$(3+1)$-Avoiding Graded Posets}
\label{sec:(3+1)}

$(3+1)$-avoidance is ubiquitous in the study of poset-avoidance, including the Stanley-Stembridge conjecture \cite{stanley-stembridge}, the birthday problem \cite{fadnavis}, etc. When we started this project, the enumeration of $(3+1)$-avoiding posets had been a long-standing unsolved problem, so we thought to explore the graded version as a stepping stone. We successfully enumerated $(3+1)$-avoiding graded posets in a standalone work joint with Lewis \cite{lewis-zhang}, with the following main result:

\begin{thm}[\cite{lewis-zhang}]
The EGF for graded $(3+1)$-avoiding posets is:
\[
F(x) = e^x - 1 + \frac{2e^{x} + (e^x - 2)\Psi_w(x, x)}{2e^{2x} + e^x + (e^{2x} - 2e^x - 1)\Psi_w(x, x)}.
\] 
\end{thm}

We do not repeat details of this work here, but we remark that the techniques we introduced played a large role:
\begin{compactitem}
\item the main strategy was, as in all our work in this paper, the transfer-matrix method;
\item to simplify the problem, we again decomposed a $(3+1)$-avoiding graded poset into a ``skeleton'' of all-seeing vertices and objects resembling faceoffs from Section~\ref{sec:(3+1)+(2+2)}, which we affectionately called \emph{teeth};
\item note the re-appearance of $\Psi_w$, which happened because we again must explore height-$2$ building blocks of our graded poset.
\end{compactitem}

This story has a happy ending -- soon afterwards, Guay-Paquet, Morales, and Rowland \cite{guay-paquet} obtained the complete enumeration of $(3+1)$-avoiding posets, using some techniques similar to ours (as expected from the intuitions which arise from our work, $\Psi_w$ appears in their work also, albeit in a different form) and some more sophisticated techniques of their own. 

\section{Conclusion and Future Directions}
\label{sec:conclusion}

In this work, we have understood the basic structure and enumeration of the four classes of graded posets and highlighted the similarities and differences between them. We hope that the reader has gained some appreciation for some of the recurring techniques and structure. There are many fruitful directions for future work:
\begin{compactitem}
\item The real value of Theorem~\ref{thm:seeds} may lie not in its modest application to generating functions, but rather in its ability to predict relationships between seemingly unrelated objects, such as the appearance of Motzkin objects, Fibonacci numbers, and directed column-convex polyominos in Figures~\ref{fig:trictionary-examples} and \ref{fig:trictionary-graded-interval-semiorders}. We welcome the reader to explore further in this direction and embark on their own garden-gathering, of the more ambitious variety than our reframing compositions as $(2+1)$-avoiding posets!
\item While the graded semiorders have a nice generating function, our enumeration in the form of matrices for all graded posets and graded interval orders still have potential to be simplified, maybe with something like a spectral theory of infinite matrices over the ring of formal power series. We also did not explore the asymptotics; for example, do graded interval orders comprise ``most'' of interval orders? The aforementioned work of Guay-Paquet et al. \cite{guay-paquet} showed that ``most'' $(3+1)$-free posets are graded.
\item Lemma~\ref{lem:locality}, while simple, is absolutely central to our work. It may be useful to try to extend it to give a unified ``Ramsey theory of chain-avoidance,'' especially if posets besides $(2+2)$ and $(3+1)$ start to become relevant in poset-avoidance.
\item It may be useful to recast the work in Section~\ref{sec:seeds}, which is really about cycle indices at heart, in a more algebraic language, such as species or invariant theory.
\end{compactitem}

\section*{Acknowledgments}
\label{sec:ack}
The author thanks Richard Stanley for motivating the enumeration of graded posets and for ever-helpful advice, and Joel Lewis, Alejandro Morales, Francois Bergeron, and Alexander Postnikov for helpful discussion. In particular, the author is especially grateful for Joel Lewis for the collaboration that led to the $(3+1)$-avoiding graded posets work, help on intuition in the other cases, and ruthless yet generous editing. Finally, two anonymous referees for the extended abstract of this work submitted to FPSAC 2015 gave very high-quality reviews and criticism. Much of this work was done at MIT, when the author was supported via an NSF Graduate Fellowship.

\appendix

\section{Primitive Seeds and Generalizations}
\label{app:seeds}

In this section, we prove Theorem~\ref{thm:seeds} and explore its underlying ideas. There are many ways to prove this result (elementary counting / Stirling numbers would suffice), but we think that using Polya theory and symmetric functions seems to make it easiest to generalize. There are many references for this classic material; Stanley \cite{stanley-ec2} would suffice.

\begin{proof}[Proof of Theorem~\ref{thm:seeds}]
Recall that the \emph{cycle index} of a $G$-action on a set $S$ is 
\[
Z_G(x_1, x_2, \ldots) = \sum_f \prod_i x_i^{|f^{-1}(c_i)|},
\]
where we sum over $G$-equivalence classes of functions $f\colon S \rightarrow C = \{c_1, c_2, \ldots \}$ for a set $C$ of ``colors.'' Suppose we have a seed $S$ and wish to obtain the generating function for the unlabeled sprouts of $S$. A sprout of $S$ is where each element of $S$ appears as a union of some nonzero number of clones. Thus, sprouts are is in bijection with ``colors'' of the elements of $S$ by by natural numbers. In our OGF, we want an element colored by $i$ to have ``weight'' $x^i$ since it corresponds to having $i$ clones of the element in the sprout. This means our OGF can be computed by plugging in $(x, x^2, \ldots)$ to $Z_G$. For a primitive seed of size $m$, our group is trivial with the cycle index $p_1^m$. As
\[p_1(x, x^2, \ldots) = x + x^2 + x^3 + \cdots = x/(1-x),\]
We now see why substitutions may be relevant: every unlabeled primitive seed of size $m$ contributes $Z_G(x, x^2, \ldots) = (\frac{x}{1-x})^m$ to the OGF of unlabeled sprouts. However, recall that the seed OGF is equal to the seed EGF $S(x)$. This means if all of our seeds are primitive and we have $s_m$ (unlabeled) seeds of each size $m$, then the OGF of the unlabeled garden satisfies
\[
O(x) = \sum_{m} s_m Z_G(x, x^2, x^3, \ldots) = \sum_m s_m (\frac{x}{1-x})^m = S(\frac{x}{1-x}).
\]
For labeled sprouts and EGFs, we want an element colored by $i$ to have weight $x^i/i!$ instead of $x^i$, so we should instead make the substitution
\[
p_1(x, x^2/2, x^3/6, \ldots) = x + x^2/2 + \cdots = e^x-1,
\]
so the EGF of the labeled garden satisfies
\[
E(x) = \sum_m s_m Z_G(x, x^2/2, x^3/6, \ldots)/m! = \sum_m (e^x-1)^m/m! = S(e^x - 1).
\]
The remaining substitutions between these $3$ functions are an exercise in simple algebra. For example, 
\[O(1-e^{-x}) = S((1-e^{-x})/e^{-x}) = S(e^x-1) = E(x).\]
This completes the proof of Theorem~\ref{thm:seeds}. 
\end{proof}

In this proof, the key insight is that all the seeds have the same automorphism group, not that the automorphism had to be trivial. This allows for a generalization. Define the \emph{faithful part} of a seed to be the elements of the seed that are not fixed by the entire automorphism group of the seed, and \emph{fixed part} otherwise. Then we have:

\begin{thm}
\label{thm:seeds generalized}
When all seeds in a garden have faithful parts isomorphic to $K$, the garden's OGF $O(x)$, EGF $E(x)$, and seed EGF $S(x)$ satisfy $E(x) = R_K(x)S(e^x-1)$, $O(x) = R_K(x)S(x/(1-x))$, and $\frac{E(x)}{R_K(x)} = \frac{O(1-e^{-x})}{R_K(1-e^{-x})},$ for a function $R_K(x)$ that depends only on $K$.
\end{thm}
\begin{proof}
Recall that a single unlabeled seed contributes 
\[
\frac{1}{|G|} \sum_{g \in G} p_{g}(x, x^2, x^3, \ldots)
\]
to the OGF for unlabeled sprouts by Polya theory, where $p_{g}$ is the power-sum symmetric function of $g$'s cycle structure as a permutation in $S_m$. If a seed of size $m$ has faithful part $K$ of size $k$, then it has fixed part of size $m-k$, meaning each $p_g$ in the above sum has a $p_1^{m-k}$ factor, which we can pull out to get 
\[
(p_1^{m-k}) (\frac{1}{|G|} \sum_{g \in G} p'_{g}),
\]
where $p'_g = p_g/p_1^{m-k}$ can be interpreted as the cycle index term where we treat $g$ as a permutation in $S_{k}$ instead of $s_m$. This means we can rewrite this further as 
\[
(p_1^m)(\frac{1}{|G|} \sum_{g \in G} \frac{p'_{g}}{p_1^k}) = p_1^m R_K(x),
\]
where $R_K(x)$ only depends on the faithful part $K$ and not on the rest of the seed! As an example, for primitive seeds, $K = \emptyset$ and $R_K(x) = 1$. So now, suppose we have a garden where all seeds have isomorphic faithful part $K$, with seed OGF $S(x) = \sum_m s_m x^n$. Then we similarly get that the OGF of unlabeled sprouts of this garden is just
\[
O(x) = \sum_m s_m p_1^m R_K(x) = R_K(x) S(\frac{x}{1-x}).
\]
A similar computation gives the EGF $E(x)$, and simple algebra gives the remaining identities. 
\end{proof}

While it seems unlikely that we will naturally come upon gardens whose seeds have isomorphic faithful parts, Theorem~\ref{thm:seeds generalized} suggests the heuristic of dividing our objects into classes based on faithful parts. It would be nice to see this method applied to other families of objects; we suspect it would be especially helpful when the types of automorphism groups on the objects are very constrained. Furthermore, while we wished to concentrate on the scope of posets in our work, the concept of posets was not fundamental to our above analysis. These results are very general concepts applicable to other combinatorial structures, probably best eventually put into the language of species. 

\section{An Application of Theorem~\ref{thm:seeds}: Marked Interval Orders}
\label{sec:marked}

We end with a generalization of a result of Stanley \cite{stanley-arrangements} from studying hyperplane arrangements. Consider an embedding of $n$ intervals in $\RR$ such that each interval is marked at several points. Consider two arrangements to be equivalent if for every two intervals $a$ and $b$, the number of marked points of $a$ to the left of the leftmost (not necessarily marked!) point of $b$ are the same. Call such an equivalence class a \emph{marked interval order}.  As an example, if each interval were only marked at the right endpoint, we would recover all interval orders of a particular set of lengths. 

These marks can be seen as a natural way to denote different levels of relationships between two intervals to give a finer relationship than just ``comparable'' or ``incomparable.'' Furthermore, marked interval orders can be put in bijection with regions of some familiar hyperplane arrangements. For example, when each interval is only marked at the right endpoint and have length $1$, we obtain semiorders as regions of the arrangement produced by the hyperplanes $x_i - x_j \in \{1, -1\}$ over all $i$ and $j$. When each interval has exactly $n$ equidistant marks, we obtain a similar situation as generalized Shi arrangements, using the hyperplanes $x_i - x_j \in \{-k, -k+1, \ldots, -1, 1, \ldots, k\}$.

Each such family $O$ of marked interval orders corresponds to a region of the hyperplane arrangement $A(O)$ of $\RR^n$ constructed by the following: for each $j$, pick $l_1, l_2, \cdots, l_n$ dependent on the interval lengths and create the hyperplane arrangement from the hyperplanes $x_i - x_j = l_1, l_2, \ldots, l_n$ for all $i$. In general, we are going to get deformations of the braid arrangement. With the tools we have developed in this work, we can generalize Theorem 2.3 from \cite{stanley-arrangements} and the result from Postnikov and Stanley \cite{postnikov-stanley} that labeled semiorders correspond (with multiplicity $n!$) to regions of the Catalan arrangement. The key is, unsurprisingly, that these objects ``avoid\footnote{When all the intervals of a marked interval order $A$ are marked in the same way, as is the case in both of our examples, it is easy to see that $A$ avoids $(3+1)$ as well. We do not prove this claim since it is irrelevant for our proof, but it is a cute curiosity.} $(2+2)$,'' with the proper definition of ``avoid.''

\begin{thm}
Let $O$ be any family of marked interval orders. Labeled members of $O$ correspond to regions of $A(O)$ and unlabeled members of $O$ correspond (with multiplicity $n!$) to regions of $A(O) \cup B$, where $B$ is the Braid arrangement; their EGF $E(x)$ and OGF $O(x)$ satisfy $E(x) = O(1-e^{-x})$.
\end{thm}
\begin{proof}
Throughout this whole proof, we use slightly stronger versions of the definitions of \emph{gardens}, \emph{seeds}, \emph{$(2+2)$-avoidance}, etc. that apply to directed acyclic graphs and not just posets. This generalization is trivial and left to the reader; simply represent each poset $P$ by its \emph{directed compatibility graph}, a directed acyclic graph where we have an edge $x \rightarrow y$ exactly when $x > y$ in $P$. Similarly, the analogue of $(2+2)$ is the directed graph on $4$ vertices with two disjoint edges. It is easy to check that for a poset, its directed compatibility graph avoids $(2+2)$ if and only if its Hasse diagram does.

Marked interval orders can be seen as posets enriched with ``nuanced'' comparability. To be precise, two of these marked intervals $a$ and $b$, count the number of marked points (on both intervals) between the left-most point $r$ in $a$ and the leftmost point $s$ in $b$. If this is $0$, then call the two  intervals \emph{incomparable.} If this is some positive number $t$, then say $a <_t b$ if $r < s$ and $b <_t a$ otherwise. This is a natural generalization of semiorders, where we just have one $x_i = 1$, and our notion of $a <_1 b$ corresponds exactly to our old notion of $a < b$. Thus, for each marked interval order $A$ we can construct an acyclic directed weighted graph $G(A)$, by assigning to each pair $a$ and $b$ of intervals a directed edge $a \to b$ with weight $t$ if $a >_t b$ and no edge if they are incomparable. Note this is not quite a generalization of the Hasse diagram; rather it is a generalization of the directed comparability graph. We could have developed the entire theory of this work with directed comparability graphs instead of the Hasse diagram; we chose not to do so as they capture the same information but the Hasse diagram is much more compact.

The key insight is that when seen as a directed graph, any $G(A)$ is basically a Hasse diagram of an interval order except with more data, and thus must avoid $(2+2)$ . Suppose $G(A)$ contains $(2+2)$. Consider the following map $\phi$: take $A$, and replace each marked interval by the leftmost interval. It is routine to check that the resulting marked interval order $\phi(A)$, which is now actually an interval order, has two intervals comparable exactly when they were comparable in $A$ (except the weight goes to $1$). One can also think of this as the forgetful functor in the appropriate category that ``forgets'' the weight assignment. Being an interval order, the result is in fact $(2+2)$-avoiding. Proposition~\ref{prop:seeds} (or rather, a tweaked version for directed acyclic graphs) goes through with no nontrivial changes.
\end{proof}

\nocite{*}
\bibliographystyle{alpha}
\bibliography{graded-posets}
\label{sec:biblio}

\end{document}